\documentclass[aap,preprint]{imsart}

\RequirePackage{amssymb,amsfonts,amsthm,amsmath}
\RequirePackage{bm,bbm}
\RequirePackage{epsfig,graphicx}
\RequirePackage[colorlinks,citecolor=blue,urlcolor=blue]{hyperref}

\arxiv{arXiv:0000.0000}

\startlocaldefs
\theoremstyle{plain}
\newtheorem{theorem}{Theorem}
\newtheorem{lemma}{Lemma}
\newtheorem{assumption}{Assumption}

\def \sd  {{\underline{s}}}
\def \td  {{\underline{t}}}
\def \ud  {{\underline{u}}}

\def \so  {{\overline{s}}}

\def \uo  {{\overline{u}}}

\def \hX  {{\widehat{X}}}
\def \bX  {{\bar{X}}}

\def \tn  {{t_n}}
\def \tnp {{t_{n+1}}}

\def \EE  {{\mathbb{E}}}
\def \PP  {{\mathbb{P}}}
\def \RR  {{\mathbb{R}}}

\def \D   {{\rm d}}
\def \DW  {{\Delta W}}

\def \halfs {{\textstyle \frac{1}{2}}}

\newcommand{\fracs}[2]{{\textstyle \frac{#1}{#2}}}

\endlocaldefs

\begin{document}

\begin{frontmatter}
\title{Adaptive Euler-Maruyama method for\\ 
       SDEs with non-globally Lipschitz drift:\\
       part I, finite time interval}
\runtitle{Adaptive Euler-Maruyama method for non-Lipschitz drift}

\begin{aug}
\author{\fnms{Wei} \snm{Fang}\ead[label=e1]{wei.fang@maths.ox.ac.uk}}
\and
\author{\fnms{Michael~B.} \snm{Giles}\ead[label=e2]{mike.giles@maths.ox.ac.uk}}

\runauthor{W.~Fang \& M.B.~Giles}

\affiliation{University of Oxford}
\address{Mathematical Institute\\University of Oxford\\Oxford OX2 6GG\\United Kingdom\\
\printead{e1}\\
\phantom{E-mail:\ }\printead*{e2}}
\end{aug}

\begin{abstract}
This paper proposes an adaptive timestep construction 
for an Euler-Maruyama approximation of SDEs with a drift 
which is not globally Lipschitz.  It is proved that 
if the timestep is bounded appropriately, then over
a finite time interval the numerical approximation is 
stable, and the expected number of timesteps is finite. 
Furthermore, the order of strong convergence is the 
same as usual, i.e.~order $\halfs$ for SDEs with a 
non-uniform globally Lipschitz volatility, and order 
$1$ for Langevin SDEs with unit volatility and a drift 
with sufficient smoothness.  The analysis is supported 
by numerical experiments for a variety of SDEs.
\end{abstract}



\end{frontmatter}

\section{Introduction}

In this paper we consider an $m$-dimensional stochastic 
differential equation (SDE) driven by a $d$-dimensional 
Brownian motion:
\begin{equation}
\D X_t = f(X_t)\,\D t + g(X_t)\,\D W_t,
\label{SDE}
\end{equation}
with a fixed initial value $X_0.$ The standard theory assumes the drift 
$f: \RR^m\!\rightarrow\!\RR^m$ 
and the volatility 
$g: \RR^m\!\rightarrow\!\RR^{m\times d}$ 
are both globally Lipschitz.  Under this 
assumption, there is well-established theory on the existence 
and uniqueness of strong solutions, and the numerical 
approximation $\hX_t$ obtained from the Euler-Maruyama 
discretisation
\[
\hX_{(n\!+\!1)h} = \hX_{nh} + f(\hX_{nh}) \, h + g(\hX_{nh}) \, \DW_n
\]
using a uniform timestep of size $h$ with Brownian increments
$\Delta W_n$, plus a suitable 
interpolation within each timestep, is known \cite{kp92}
to have a strong error which is $O(h^{1/2})$ so that 
for any $T, p > 0$
\[
\EE\left[\sup_{0\leq t \leq T}\|\hX_t \!-\! X_t \|^p\right] = O(h^{p/2}).
\]

The interest in this paper is in other cases in which $g$ is 
again globally Lipschitz, but $f$ is only locally Lipschitz.
If, for some $\alpha, \beta \geq 0$,  $f$ also
satisfies the one-sided growth condition\\[-0.2in]
\[
\langle x,f(x)\rangle \leq \alpha\|x\|^2+\beta,
\]
where $\langle \cdot, \cdot \rangle$ denotes an inner product,
then it is again possible to prove the existence and uniqueness 
of strong solutions (see Theorems 2.3.5 and 2.4.1 in \cite{mao97}).  
Furthermore (see Lemma 3.2 in \cite{hms02}), these solutions 
are stable in the sense that for any $T,p > 0$\\[-0.1in]
\[
\EE\left[\sup_{0\leq t \leq T}\| X_t \|^p\right] < \infty.
\]
The problem is that the numerical approximation given by 
the uniform timestep Euler-Maruyama discretisation
may not be stable.  Indeed, for the SDE
\begin{equation}
\D X_t = - X_t^3 \,\D t + \D W_t,
\label{eq:cubic_drift}
\end{equation}
it has been proved \cite{hjk11} that for any 
$T\!>\!0$ and $p\!\geq\!2$
\[
\lim_{h\rightarrow 0} \EE\left[ \| \hX_T \|^p \right] = \infty.
\]

This behaviour has led to research on numerical methods which 
achieve strong convergence for these SDEs with a non-globally 
Lipschitz drift.  One key paper in this area is by 
Higham, Mao \& Stuart \cite{hms02}.  First, assuming a locally 
Lipschitz condition for both the drift and the volatility, 
they prove that if the uniform timestep Euler-Maruyama 
discretisation is stable then it also converges strongly.
Assuming the drift satisfies a one-sided Lipschitz 
condition and a polynomial growth condition, they then prove 
stability and the standard order $\halfs$ 
strong convergence for two uniform timestep implicit methods, 
the Split-Step Backward Euler method (SSBE):
\begin{eqnarray*}
\hX_{nh}^*     &=& \hX_{nh} + f(\hX_{nh}^*)\, h\\
\hX_{(n\!+\!1)h} &=& \hX_{nh}^* + g(\hX_{nh}^*)\, \Delta W_n.
\end{eqnarray*}
and the drift-implicit Backward Euler method:
\[
\hX_{(n\!+\!1)h} = \hX_{nh} + f(\hX_{(n\!+\!1)h})\, h + g(\hX_{nh})\, \Delta W_n.
\]
Mao \& Szpruch \cite{ms13} prove that the implicit $\theta$-Euler method
\[
\hX_{(n\!+\!1)h} = \hX_{nh} + \theta f(\hX_{(n\!+\!1)h})\, h
        + (1\!-\!\theta) f(\hX_{nh})\, h + g(\hX_{nh})\, \Delta W_n,
\]
converges strongly for $\halfs\!\leq\!\theta\!\leq\!1$ 
under more general conditions which permit a non-globally Lipschitz volatility.

However, except for some special cases, implicit methods can
require significant additional computational costs, especially
for multi-dimensional SDEs; therefore, a stable explicit method 
is desired. 
Milstein \& Tretyakov proposed a general approach
which discards approximate paths that cross a sphere with
a sufficiently large radius $R$ \cite{mt05}.  However,  it is 
not easy to quantify the errors due to $R$.
The explicit tamed Euler method proposed by Hutzenthaler, 
Jentzen \& Kloeden \cite{hjk12} is
\[
\hX_{(n\!+\!1)h} = \hX_{nh} + \frac{ f(\hX_{nh})}{1\!+\!C\,h\|f(\hX_{nh})\|}\, h
               + g(\hX_{nh})\, \DW_n,
\]
for some fixed constant $C\!>\!0$.  They prove both stability 
and the standard order $\halfs$ strong convergence.  This 
approach has been extended to the tamed Milstein method by 
Wang \& Gan \cite{wg13}, proving order 1 strong convergence
for SDEs with commutative noise.
Finally, Mao \cite{mao15} proposes a truncated Euler method which
has the form
\begin{eqnarray*}
\hX_{(n\!+\!1)h} &=& \hX_{nh} 
+ f\left( \min( K \|\hX_{nh}\|^{-1}\!, 1 )\, \hX_{nh}\right) \, h \\
&& ~~~~
+ g\left( \min( K \|\hX_{nh}\|^{-1}\!, 1 )\, \hX_{nh}\right) \, \DW_n.
\end{eqnarray*}
By making $K$ a function of $h$, strong convergence is proved 
for SDEs satisfying a Khasminskii-type condition which again
allows a non-globally Lipschitz volatility;
in \cite{mao16} it is proved that the order of 
convergence is arbitrarily close to $\halfs$.

In this paper, we propose instead to use the standard explicit 
Euler-Maruyama method, but with an adaptive timestep $h_n$ which 
is a function of the current approximate solution $\hX_{t_n}$.
The idea of using an adaptive timestep comes from considering 
the divergence of the uniform timestep method for the SDE 
(\ref{eq:cubic_drift}). When there is no noise, the requirement
for the explicit Euler approximation of the corresponding ODE 
to have a stable monotonic decay is that its timestep satisfies 
$h\!<\!\hX_{t_n}^{-2}$.  An intuitive explanation for the 
instability of the uniform timestep Euler-Maruyama approximation
of the SDE is that there is always a very small probability of
a large Brownian increment $\DW_n$ which pushes the 
approximation $\hX_{t_{n+1}}$ into the region $h\!>\!2\,\hX_{t_{n+1}}^{-2}$
leading to an oscillatory super-exponential growth.  Using
an adaptive timestep avoids this problem.

Adaptive timesteps have been used in previous research to improve 
the accuracy of numerical approximations.  Some approaches use
local error estimation to decide whether or not to refine the 
timestep \cite{gl97,mauthner98,lamba03} while others are similar 
to ours in setting the size of each timestep based on the current 
path approximation $\hX_t$ \cite{hmr01,muller04}.  However,
these all assume globally Lipschitz drift and volatility.
The papers by Lamba, Mattingly \& Stuart \cite{lms07} and
Lemaire \cite{lemaire07} are more relevant to the analysis in 
this paper.  They both consider drifts which are not globally 
Lipschitz, but they assume a dissipative condition which is 
stronger than the conditions assumed in this paper.  Lamba, 
Mattingly \& Stuart \cite{lms07} prove strong stability but not 
the order of strong convergence, while Lemaire \cite{lemaire07}
considers an infinite time interval with a timestep with an 
upper bound which decreases towards zero over time, and proves 
convergence of the empirical distribution to the invariant 
distribution of the SDE.

In this paper we are concerned with strong convergence, not 
weak convergence, because our interest is in using the numerical 
approximation as part of a multilevel Monte Carlo (MLMC) 
computation \cite{giles08,giles15} for which the strong 
convergence properties are key in establishing the rate of 
decay of the variance of the multilevel correction.  
Usually, MLMC is used with a geometric sequence of time grids, 
with each coarse timestep corresponding to a fixed number of 
fine timesteps.  However, it has been shown that it 
is not difficult to implement MLMC using the same driving 
Brownian path for the coarse and fine paths, even when they 
have no time points in common \cite{glw16}.

Paper \cite{glw16} also provides another motivation for this
paper, the analysis of Langevin equations with a drift 
$-\nabla V(X_t)$ where $V(x)$ is a potential function which
comes from the modelling of molecular dynamics. \cite{glw16}
considers the FENE (Finitely Extensible Nonlinear Elastic) model
which in the case of a molecule with a single bond has a 
3D potential $-\mu \log (1 \!-\! \|x\|^2)$.
Considerations of stability and accuracy lead to the use 
of a timestep of the form
$ 
\delta\, (1 \!-\! \|\hX_n\|)^2/\max(2\mu, 36),
$
for some $0\!<\!\delta\!\leq \!1$.  
Because of this, we pay particular attention to the case of 
Langevin equations, and for these we prove first order strong 
convergence, the same as for the uniform timestep Euler-Maruyama 
method for globally Lipschitz drifts.  Unfortunately our
assumptions do not cover the case of the FENE model as we 
require $-\nabla V(x)$ to be locally Lipschitz on $\RR^m$.

The rest of the paper is organised as follows.  Section 2 states
the main theorems, and proves some minor lemmas.  Section 3
has a number of example applications, many from \cite{hj15},
illustrating how suitable adaptive timestep functions can be defined.  
It also presents some numerical results comparing the performance of 
the adaptive Euler-Maruyama method to other methods.  Section 4 has 
the proofs of the three main theorems, and finally Section 5 has some 
conclusions and discusses the extension to the infinite time interval 
which will be covered in a future paper.

In this paper we assume the following setting and notation.
Let $T\!>\!0$ be a fixed positive real number, and let 
$(\Omega,\mathcal{F},\PP)$ be a probability space with 
normal filtration $(\mathcal{F}_t)_{t\in[0,T]}$ corresponding to
a $d$-dimensional standard Brownian motion 
$W_t=(W^{(1)},W^{(2)},\ldots,W^{(d)})_t^T$.
We denote the vector norm by 
$\|v\|\triangleq(|v_1|^2+|v_2|^2+\ldots+|v_m|^2)^{\frac{1}{2}}$, 
the inner product of vectors $v$ and $w$ by 
$\langle v,w \rangle\triangleq v_1w_1+v_2w_2+\ldots+v_mw_m$, 
for any $v,w\in\RR^m$ and the Frobenius matrix norm by 
$\|A\|\triangleq \sqrt{\sum_{i,j}A_{i,j}^2}$ 
for all $A\in\RR^{m\times d}.$

\section{Adaptive algorithm and theoretical results}

\subsection{Adaptive Euler-Maruyama method}

The adaptive Euler-Maruyama discretisation is
\[
\tnp = \tn + h_n, ~~~~
\hX_\tnp = \hX_\tn + f(\hX_\tn)\, h_n + g(\hX_\tn)\, \DW_n,
\]
where 
$h_n\triangleq h(\hX_\tn)$ and $\DW_n \triangleq W_\tnp \!-\! W_\tn$,
and there is fixed initial data $t_0\!=\!0,\ \hX_0\!=\!X_0$.

One key point in the analysis is to prove that $t_n$ increases 
without bound as $n$ increases.  More specifically, the analysis 
proves that for any $T\!>\!0$, almost surely for each path 
there is an $N$ such that $t_N\!\geq\!T$.

We use the notation 
$
\td \triangleq \max\{t_n: t_n\!\leq\! t\},\
n_t\triangleq\max\{n: t_n\!\leq\! t\}
$
for the nearest time point before time $t$, and its index.

We define the piecewise constant interpolant process $\bX_t=\hX_\td$ 
and also define the standard continuous interpolant \cite{kp92} as
\[
\hX_t=\hX_\td+f(\hX_\td) (t\!-\!\td) + g(\hX_\td) (W_t\!-\!W_\td),
\]
so that $\hX_t$ is the solution of the SDE
\begin{equation}
\D \hX_t =  f(\hX_\td)\, \D t + g(\hX_\td) \, \D W_t = f(\bX_t)\,\D t + g(\bX_t)\, \D W_t.
\label{eq:approx_SDE}
\end{equation}

In the following subsections, we state the key results on 
stability and strong convergence, and related results on 
the number of timesteps, introducing various assumptions 
as required for each.  The main proofs are deferred to 
Section \ref{sec:proofs}.

\subsection{Stability}

\begin{assumption}[Local Lipschitz and linear growth]
\label{assp:linear_growth}
$f$ and $g$ are both locally Lipschitz, so that for any 
$R\!>\!0$ there is a constant $C_R$ such that
\[
\| f(x)\!-\!f(y) \| + \| g(x)\!-\!g(y) \| \leq C_R\, \|x\!-\!y\| 
\]
for all $x,y\in \RR^m$ with $\|x\|,\|y\|\leq R$.
Furthermore, there exist constants $\alpha, \beta \geq 0$ 
such that for all $x\in \RR^m$, $f$ satisfies the one-sided 
linear growth condition:
\begin{equation}
\langle x,f(x)\rangle \leq \alpha\|x\|^2+\beta,
\label{eq:onesided_growth}
\end{equation}
and $g$ satisfies the linear growth condition:
\begin{equation}
\|g(x)\|^2\leq \alpha\|x\|^2+\beta.
\label{eq:g_growth}
\end{equation} 
\end{assumption}

Together, (\ref{eq:onesided_growth}) and (\ref{eq:g_growth}) 
imply the monotone condition
\[
\langle x,f(x)\rangle + \halfs \|g(x)\|^2 \leq 
\fracs{3}{2} ( \alpha\|x\|^2\!+\!\beta),
\]
which is a key assumption in the analysis of Mao \& Szpruch 
\cite{ms13} and Mao \cite{mao15} for SDEs with volatilities 
which are not globally Lipschitz. 
However, in our analysis we choose to use this slightly 
stronger assumption, which provides the basis for the following 
lemma on the stability of the SDE solution.

\begin{lemma}[SDE stability]

\label{lemma:SDE_stability}

If the SDE satisfies Assumption \ref{assp:linear_growth},
then for all $p\!>\!0$ 
\[
\EE\left[ \, \sup_{0\leq t\leq T} \|X_t\|^p \right] < \infty.
\]
\end{lemma}

\begin{proof}
The proof is given in Lemma 3.2 in \cite{hms02}; the statement of 
that lemma makes stronger assumptions on $f$ and $g$, corresponding 
to (\ref{eq:onesided_Lipschitz}) and (\ref{eq:g_Lipschitz}), but the 
proof only uses the conditions in Assumption \ref{assp:linear_growth}.
\end{proof}

\vspace{0.1in}

We now specify the critical assumption about the adaptive timestep.

\begin{assumption}[Adaptive timestep] 
\label{assp:timestep}

The adaptive timestep function $h: \RR^m\rightarrow \RR^+$ is 
continuous and strictly positive, and there exist constants 
$\alpha, \beta > 0$ such that for all 
$x\in \RR^m$, $h(x)$ satisfies the inequality
\begin{equation}
\langle x, f(x) \rangle + \halfs \, h(x)\, \|f(x)\|^2 \leq \alpha \|x\|^2 + \beta.
\label{eq:timestep}
\end{equation}
\end{assumption}

Note that if another timestep function $h^\delta(x)$ is smaller than 
$h(x)$, then  $h^\delta(x)$ also satisfies the Assumption \ref{assp:timestep}.  Note 
also that the form of (\ref{eq:timestep}), which is motivated by 
the requirements of the proof of the next theorem, is very similar 
to (\ref{eq:onesided_growth}).  Indeed, if (\ref{eq:timestep}) is 
satisfied then (\ref{eq:onesided_growth}) is also true for the 
same values of $\alpha$ and $\beta$.

\begin{theorem}[Finite time stability]

\label{thm:stability}

If the SDE satisfies Assumption \ref{assp:linear_growth}, and the 
timestep function $h$ satisfies Assumption \ref{assp:timestep},
then $T$ is almost surely attainable
(i.e.~for $\omega\in\Omega$, 
$\PP(\exists N(\omega)<\infty~\mbox{s.t.}~t_{N(\omega)}\!\geq\!T)=1$) 
and for all $p\!>\!0$ there exists a constant $C_{p,T}$ which 
depends solely on $p$, $T$ and the constants $\alpha, \beta$ in
Assumption \ref{assp:timestep}, such that
\[
\EE\left[ \, \sup_{0\leq t\leq T} \|\hX_t\|^p \right] < C_{p,T}.
\]
\end{theorem}

\begin{proof}
The proof is deferred to Section \ref{sec:proofs}.
\end{proof}

To bound the expected number of timesteps, we require an assumption 
on how quickly $h(x)$ can approach zero as $\|x\|\rightarrow\infty$.

\begin{assumption}[Timestep lower bound]
\label{assp:lower_bound}
There exist constants $\xi,\zeta,q\!>\!0$, such that the adaptive timestep 
function satisfies the inequality
\[
h(x) \geq \left( \xi \|x\|^q + \zeta \right)^{-1}.
\]
\end{assumption}

Given this assumption, we obtain the following lemma.

\begin{lemma}[Bounded timestep moments]

If the SDE satisfies Assumption \ref{assp:linear_growth}, and the 
timestep function $h$ satisfies Assumptions \ref{assp:timestep}
and \ref{assp:lower_bound}, then for all $p\!>\!0$
\[
\EE\left[ N_T^p \right] < \infty.
\]
where
$\displaystyle
N_T = \min\{n: t_n\geq T\}
$
is the number of timesteps required by a path approximation.
\end{lemma}

\begin{proof}
Assumption \ref{assp:lower_bound} gives us
\[
N_T 
\leq 1 + T \sup_{0\leq t\leq T} \frac{1}{h(\hX_t)}
\leq 1 + T \left( \xi \sup_{0\leq t\leq T} \|\hX_t\|^q + \zeta \right),
\]
and the result is then an immediate consequence of 
Theorem \ref{thm:stability}.
\end{proof}

\subsection{Strong convergence}

Standard strong convergence analysis for an approximation 
with a uniform timestep $h$ considers the limit $h\!\rightarrow\!0$.
This clearly needs to be modified when using an adaptive timestep,
and we will instead consider a timestep function $h^\delta(x)$ controlled
by a scalar parameter $0\!<\!\delta\!\leq\!1$, and consider the
limit $\delta\!\rightarrow\!0$.

Given a timestep function $h(x)$ which satisfies Assumptions 
\ref{assp:timestep} and \ref{assp:lower_bound}, ensuring stability as 
analysed in the previous section, there are two quite natural ways in 
which we might introduce $\delta$ to define $h^\delta(x)$:
\begin{eqnarray*}
&& h^\delta(x) = \delta\, \min(T, h(x) ), \\
&& h^\delta(x) = \min( \delta\,T, h(x) ).
\end{eqnarray*}
The first refines the timestep everywhere, while the latter 
concentrates the computational effort on reducing the 
maximum timestep, with $h(x)$ introduced to ensure stability
when $\|\hX_t\|$ is large.

In our analysis, we will cover both possibilities by making the 
following assumption.

\begin{assumption}
\label{assp:timestep_delta}

The timestep function $h^\delta$, satisfies the inequality
\begin{equation}
\label{eq:h_delta}
\delta\, \min(T, h(x)) \leq h^\delta(x) \leq \min( \delta\,T, h(x) ),
\end{equation}
and $h$ satisfies Assumption \ref{assp:timestep}.
\end{assumption}

Given this assumption, we obtain the following theorem:

\begin{theorem}[Strong convergence]

\label{thm:convergence}
If the SDE satisfies Assumption \ref{assp:linear_growth}, and the timestep 
function $h^\delta$ satisfies Assumption \ref{assp:timestep_delta},
then for all $p\!>\!0$
\[
\lim_{\delta\rightarrow 0}\EE\left[ \, \sup_{0\leq t\leq T} \|\hX_t\!-\!X_t\|^p \right] = 0.
\]
\end{theorem}

\begin{proof}
The proof is essentially identical to the uniform timestep 
Euler-Maruyama analysis in Theorem 2.2 in \cite{hms02} by
Higham, Mao \& Stuart.  

The only change required by the use of an adaptive timestep 
is to note that
\[
\hX_s - \bX_s = f(\bX_s) \, (s\!-\!\sd) + g(\bX_s) \, (W_s\!-\!W_\sd) 
\]
and $s\!-\!\sd < \delta \, T$ and 
$
\EE\left[\, \|W_s\!-\!W_\sd\|^2 \ | \ {\cal F}_\sd \,\right] = d \, (s\!-\!\sd).
$
\end{proof}

To prove an order of strong convergence requires new assumptions 
on $f$ and $g$:

\begin{assumption}[Lipschitz properties]

\label{assp:Lipschitz}

There exists a constant $\alpha\!>\!0$ such that for all $x,y\in\RR^m$,
$f$ satisfies the one-sided Lipschitz condition:
\begin{equation}
\langle x\!-\!y,f(x)\!-\!f(y)\rangle \leq \halfs\alpha\|x\!-\!y\|^2,
\label{eq:onesided_Lipschitz}
\end{equation}
and $g$ satisfies the Lipschitz condition:
\begin{equation}
\|g(x)\!-\!g(y)\|^2\leq \halfs\alpha \|x\!-\!y\|^2.
\label{eq:g_Lipschitz}
\end{equation} 
In addition, $f$ satisfies the locally polynomial growth Lipschitz condition
\begin{equation}
\|f(x)\!-\!f(y)\| \leq \left(\gamma\, (\|x\|^q \!+\! \|y\|^q) + \mu\right) \, \|x\!-\!y\|,
\label{eq:local_Lipschitz}
\end{equation}
for some $\gamma, \mu, q > 0$.
\end{assumption}

Note that setting $y\!=\!0$ gives
\[
\langle x, f(x) \rangle 
\ \leq\ \halfs \alpha \|x\|^2 + \langle x, f(0) \rangle
\ \leq\        \alpha \|x\|^2 + \halfs \alpha^{-1} \|f(0)\|^2,
\]
\[
\| g(x) \|^2 
\ \leq\  2 \|g(x)\!-\!g(0)\|^2 + 2 \|g(0)\|^2
\ \leq\ \alpha \|x\|^2 +  2 \|g(0)\|^2.
\]
Hence, Assumption \ref{assp:Lipschitz} implies Assumption 
\ref{assp:linear_growth}, with the same $\alpha$ and 
an appropriate $\beta$.

Also, if the drift and volatility are differentiable, 
the following assumption is equivalent to Assumption 
\ref{assp:Lipschitz}, and usually easier to check in practice.

\begin{assumption}[Lipschitz properties]

\label{assp:Lipschitz_diff}

There exists a constant $\alpha\!>\!0$ such that for 
all $x,e\in\RR^m$ with $\|e\|\!=\!1$,
$f$ satisfies the one-sided Lipschitz condition:
\begin{equation}
\langle e ,\nabla f(x)\, e \rangle \leq \halfs\alpha
\label{eq:onesided_Lipschitz_diff}
\end{equation}
and $g$ satisfies the Lipschitz condition:
\begin{equation}
\| \nabla g(x)\|^2\leq \halfs\alpha
\label{eq:g_Lipschitz_diff}
\end{equation} 
and in addition $f$ satisfies the locally polynomial growth 
Lipschitz condition
\begin{equation}
\|\nabla f(x) \| \leq \left( 2\,\gamma\, \|x\|^q + \mu\right),
\label{eq:local_Lipschitz_diff}
\end{equation}
for some $\gamma, \mu, q > 0$.
\end{assumption}

\begin{theorem}[Strong convergence order]

\label{thm:convergence_order}

If the SDE satisfies Assumption \ref{assp:Lipschitz}, and the timestep 
function $h^\delta$ satisfies Assumption \ref{assp:timestep_delta},
then for all $p\!>\!0$ there exists a constant $C_{p,T}$ such that
\[
\EE\left[ \, \sup_{0\leq t\leq T} \|\hX_t\!-\!X_t\|^p \right] \leq C_{p,T} \, \delta^{p/2}.
\]
\end{theorem}

\begin{proof}
The proof is deferred to Section \ref{sec:proofs}.
\end{proof}

\begin{lemma}[Number of timesteps]

\label{lemma:timesteps2}

If the SDE satisfies Assumption \ref{assp:Lipschitz}, and the timestep 
function $h^\delta(x)$ satisfies  Assumption \ref{assp:timestep_delta},
with $h(x)$ satisfying Assumption \ref{assp:lower_bound}, 
then for all $p\!>\!0$ there exists a constant $c_{p,T}$ such that
\[
\EE\left[ N_T^p \right] \leq c_{p,T} \, \delta^{-p}.
\]
where $N_T$ is again the number of timesteps required by a path approximation.
\end{lemma}

\begin{proof}
Equation (\ref{eq:h_delta}) and Assumption \ref{assp:lower_bound} give
\begin{eqnarray*}
N_T &\leq& 1 + T \sup_{0\leq t\leq T} \frac{1}{h^\delta(\hX_t)}
\\  &\leq& 1 + \delta^{-1}\, T \sup_{0\leq t\leq T} \max(h^{-1}(x),T^{-1})
\\  &\leq& \delta^{-1}\, T \left( \xi \sup_{0\leq t\leq T} \|\hX_t\|^q
 + \zeta + (1\!+\!\delta) \, T^{-1} \right).
\end{eqnarray*}
The result is then a consequence of Theorem \ref{thm:stability} 
since $h^\delta(x)\!\leq\!h(x)$ and therefore $h^\delta(x)$ satisfies 
the requirements for stability.
\end{proof}

The conclusion from Theorem \ref{thm:convergence_order} and 
Lemma \ref{lemma:timesteps2} is that 
\[
\EE\left[ \, \sup_{0\leq t\leq T} \|\hX_t\!-\!X_t\|^p \right]^{1/p} \leq 
C^{1/p}_{p,T}\, c^{1/2}_{1,T} \, (\EE\left[ N_T \right])^{-1/2}
\]
which corresponds to order $\halfs$ strong convergence when
comparing the accuracy to the expected cost.

First order strong convergence is achievable for Langevin SDEs
in which $m\!=\!d$ and $g$ is the identity matrix $I_m$, but 
this requires stronger assumptions on the drift $f$.

\begin{assumption}[Enhanced Lipschitz properties]

\label{assp:enhanced_Lipschitz}

There exists a constant $\alpha\!>\!0$ such that for all $x,y\in\RR^m$,
$f$ satisfies the one-sided Lipschitz condition:
\begin{equation}
\langle x\!-\!y,f(x)\!-\!f(y)\rangle \leq \halfs\alpha\|x\!-\!y\|^2.
\label{eq:onesided_Lipschitz2}
\end{equation}
In addition, $f$ is differentiable, and $f$ and $\nabla f$ 
satisfy the locally polynomial growth Lipschitz condition
\begin{equation}
\|f(x)\!-\!f(y)\| + \|\nabla f(x)\!-\!\nabla f(y)\|\leq 
\left( \gamma\, (\|x\|^q \!+\! \|y\|^q) + \mu\right) \|x\!-\!y\|,
\label{eq:local_Lipschitz2}
\end{equation}
for some $\gamma, \mu, q > 0$.
\end{assumption}

\begin{lemma}
\label{lemma:useful}
If $f$ satisfies Assumption \ref{assp:enhanced_Lipschitz}
then for any $x,y,v \in \RR^m$,\\[-0.1in]
\[
\langle v, f(x)\!-\!f(y) \rangle = 
\langle v, \nabla f(x) (x\!-\!y) \rangle + R(x,y,v),
\]
where the remainder term has the bound
\[
|R(x,y,v)| \leq \left( \gamma\, (\|x\|^q \!+\! \|y\|^q) + \mu\right) \|v\| \, \|x\!-\!y\|^2.
\]
\end{lemma}
\begin{proof}
If we define the scalar function $u(\lambda)$ for $0\!\leq\!\lambda\!\leq\!1$ by
\[
u(\lambda) = \langle v, f(y+\lambda(x\!-\!y)) \rangle,
\]
then $u(\lambda)$ is continuously differentiable, and by the Mean Value Theorem
$u(1)\!-\!u(0)= u'(\lambda^*)$ for some $0\!<\!\lambda^*\!<\!1$, which implies 
that
\[
\langle v, f(x)\!-\!f(y) \rangle = \langle v, \nabla f(y+\lambda^*(x\!-\!y))\, (x\!-\!y) \rangle.
\]
The final result then follows from the Lipschitz property of $\nabla f$.
\end{proof}

We now state the theorem on improved strong convergence.

\begin{theorem}[Strong convergence for Langevin SDEs]

\label{thm:convergence_order2}

If $m\!=\!d$, $g\equiv I_m$, $f$ satisfies 
Assumption \ref{assp:enhanced_Lipschitz}, 
and the timestep function $h^\delta$ satisfies Assumption 
\ref{assp:timestep_delta}, then for all $T,p\in(0,\infty)$ 
there exists a constant $C_{p,T}$ such that
\[
\EE\left[ \, \sup_{0\leq t\leq T} \|\hX_t\!-\!X_t\|^p \right] \leq C_{p,T} \, \delta^p.
\]
\end{theorem}

\begin{proof}
The proof is deferred to Section \ref{sec:proofs}.
\end{proof}

Comment: first order strong convergence can also be achieved for a 
general $g(x)$ by using an adaptive timestep Milstein discretisation, 
provided $\nabla g$ satisfies an additional Lipschitz condition. 
A formal statement and proof of this is omitted as it requires a 
lengthy extension to the stability analysis.  In addition, this 
numerical approach is only practical in cases in which the 
commutativity condition is satisfied and therefore there is no 
need to simulate the L{\'evy} areas which the Milstein method 
otherwise requires \cite{kp92}.


\section{Examples and numerical results}

In this section we discuss a number of example SDEs with 
non-globally Lipschitz drift.
In each case we comment on the applicability of the theory 
and a suitable choice for the adaptive timestep.  

We then present numerical results for three testcases which 
illustrate some key aspects.

\subsection{Scalar SDEs}

In each of the cases to be presented, the drift is of the form
\begin{equation}
\label{eq:scalar_drift}
f(x) \approx - \, c \, \mbox{sign}(x) \, |x|^q, ~~~ \mbox{as} ~ |x|\rightarrow \infty
\end{equation}
for some constants $c\!>\!0$, $q\!>\!1$.  Therefore, as
$|x|\!\rightarrow\! \infty$, the maximum stable timestep satisfying 
Assumption \ref{assp:timestep} corresponds to
$\langle x, f(x) \rangle + \halfs h(x) \, |f(x)|^2 \approx 0$
and hence $h(x) \approx 2 |x|/|f(x)| \approx 2\, c^{-1} |x|^{1-q}$.  
A suitable choice for $h(x)$ and $h^\delta(x)$ is therefore
\begin{equation}
h(x) = \min\left(T, c^{-1} |x|^{1-q}\right), ~~~
h^\delta(x) = \delta\, h(x).
\label{eq:scalar_timestep}
\end{equation}

\subsubsection{Stochastic Ginzburg-Landau equation}

This describes a phase transition from the theory of superconductivity
\cite{hjk11,kp92}.
\[
\D X_t= \left( (\eta+\halfs\sigma^2) X_t - \lambda X_t^3 \right)\D t
      + \sigma X_t\,\D W_t,
\]
where $\eta\!\geq\!0$, $\lambda,\sigma\!>\!0$. The SDE is usually 
defined on the domain $\RR^+$, since $X_t\!>\!0$ for all $t\!>\!0$, 
if $X_0\!>\!0$.  However, the numerical approximation is not guaranteed
to remain strictly positive and the domain can be extended to $\RR$ 
without any change to the SDE.

The drift and volatility satisfy Assumptions \ref{assp:linear_growth} 
and \ref{assp:Lipschitz}, and therefore all of the theory is applicable,
with a suitable choice for $h^\delta(x)$, based on (\ref{eq:scalar_drift}) 
and (\ref{eq:scalar_timestep}), being
\[
h^\delta(x) = \delta \, \min\left(T, \lambda^{-1} x^{-2}\right).
\]

\subsubsection{Stochastic Verhulst equation}

This is a model for a population with competition between 
individuals \cite{hjk11}.
\[
\D X_t=\left(\left( \eta+\halfs\sigma^2\right)X_t-\lambda X_t^2 \right) \D t 
+ \sigma X_t \, \D W_t,
\]
where $\eta,\lambda,\sigma\!>\!0.$  The SDE is defined on the domain $\RR^+$,
but can be extended to $\RR$ by modifying it to
\[
\D X_t=\left(\left( \eta+\halfs\sigma^2\right)X_t-\lambda\, |X_t| X_t \right) \D t 
+ \sigma X_t \, \D W_t,
\]
so that the drift is positive in the limit $x\!\rightarrow\!-\infty$.

The drift and volatility then satisfy Assumptions \ref{assp:linear_growth} 
and \ref{assp:Lipschitz}, and therefore all of the theory is applicable,
with a suitable choice for $h^\delta(x)$, based on (\ref{eq:scalar_drift}) 
and (\ref{eq:scalar_timestep}), being
\[
h^\delta(x) = \delta \, \min\left(T, \lambda^{-1} |x|^{-1}\right).
\]

\subsection{Multi-dimensional SDEs}

With multi-dimensional SDEs there are two cases of particular 
interest. For SDEs with a drift which, 
for some $\beta\!>\!0$ and sufficiently large $\|x\|$,
satisfies the condition
\[
\langle x, f(x) \rangle \leq - \beta\, \|x\| \, \|f(x)\|,
\]
one can take 
$\langle x, f(x) \rangle + \halfs h(x) \, |f(x)|^2 \approx 0$
and therefore a suitable definition of $h(x)$ for large $\|x\|$ 
is
\[
h(x) = \min(T, \|x\|/\|f(x)\|).
\]
For SDEs with a drift which does not satisfy the condition,
but for which $\|f(x)\|\rightarrow \infty$ as $\|x\|\rightarrow \infty$,
an alternative choice for large $\|x\|$ is to use
\begin{equation}
h(x) = \min(T, \gamma\, \|x\|^2/\|f(x)\|^2),
\label{eq:multi_h}
\end{equation}
for some $\gamma\!>\!0$.  The difficulty in this case is 
choosing the best value for $\gamma$, taking into account 
both accuracy and cost.

\subsubsection{Stochastic van der Pol oscillator}

This describes state oscillation \cite{hj15}.
\begin{eqnarray*}
\D X_t^{(1)} &=& X_t^{(2)}\,\D t\\
\D X_t^{(2)} &=& 
\left(\alpha\left(\mu\!-\!(X_t^{(1)})^2 \right) X_t^{(2)}-\delta X_t^{(1)} \right)
\,\D t + \beta \,\D W_t 
\end{eqnarray*}
where $\alpha, \mu, \delta, \beta\!>\!0$.
It can be put in the standard form by defining
\[
f(x) \equiv 
\begin{pmatrix}
x_2\\
\alpha\,(\mu\!-\!x_1^2)\,x_2-\delta x_1
\end{pmatrix},\ \ 
g(x) \equiv 
\begin{pmatrix}
0\\
\beta
\end{pmatrix}.
\]
It follows that
\[
\left\langle x,f(x) \right\rangle
\ =\ - \alpha\, x_1^2\, x_2^2 + \alpha\,\mu\, x_2^2 + (1\!-\!\delta)\, x_1\, x_2
\ \leq\ \left( \alpha \mu + \halfs (1\!-\!\delta) \right) \|x\|^2.
\]
Therefore the drift and volatility satisfy Assumption 
\ref{assp:linear_growth} and the numerical approximations 
will be stable if the maximum timestep is defined 
by (\ref{eq:multi_h}).

However, it can be verified that $\langle e, \nabla f(x) \, e\rangle$
is not uniformly bounded for an arbitrary $e$ such that $\|e\|=1$, and 
therefore the drift does not satisfy the one-sided Lipschitz condition.  
Hence the stability and strong convergence theory in this paper is 
applicable, but not the theorems on the order of convergence.
Nevertheless, numerical experiments exhibit first order strong 
convergence, which is consistent with the fact that the volatility 
in uniform, so it seems there remains a gap here in the theory.

\subsubsection{Stochastic Lorenz equation}

This is a three-dimensional system modelling convection rolls 
in the atmosphere \cite{hj15}.
\begin{eqnarray*}
\D X_t^{(1)} &=& \left(\alpha_1X_t^{(2)}-\alpha_1X_t^{(1)}\right) \,\D t
                + \beta_1X_t^{(1)}\D W_t^{(1)}\\
\D X_t^{(2)} &=& \left(\alpha_2X_t^{(1)}-X_t^{(2)}-X_t^{(1)}X_t^{(3)}\right) \,\D t
                + \beta_2X_t^{(2)}\D W_t^{(2)}\\
\D X_t^{(3)}&=& \left(X_t^{(1)}X_t^{(2)}-\alpha_3X_t^{(3)}\right) \,\D t
                + \beta_3X_t^{(3)}\D W_t^{(3)}
\end{eqnarray*}
where $\alpha_1,\alpha_2,\alpha_3,\beta_1,\beta_2,\beta_3 > 0$,
and so we have:
\[
f(x) \equiv 
\begin{pmatrix}
\alpha_1(x_2-x_1)\\
\alpha_2 x_1-x_2-x_1x_3\\
x_1x_2-\alpha_3x_3
\end{pmatrix}, ~~
g(x) \equiv
\begin{pmatrix}
\beta_1x_1&0&0\\
0&\beta_2x_2&0\\
0&0&\beta_3x_3
\end{pmatrix}.
\]
The diffusion coefficient is globally Lipschitz, and since 
$\langle x, f(x) \rangle$ consists solely of quadratic terms,
the drift satisfies the one-sided linear growth condition.
Noting that $\|f\|^2 \approx x_1^2(x_2^2+x_3^2) < \|x\|^4 $ as
$\|x\|\rightarrow\infty$, an appropriate maximum timestep is
\[
h(x) = \min(T, \gamma \|x\|^{-2}),
\]
for any $\gamma\!>\!0$.
However, the drift does not satisfy the one-sided Lipschitz 
condition, and therefore the theory on the order of strong 
convergence is not applicable.

\subsubsection{Langevin equation}

The multi-dimensional Langevin equation is
\begin{equation}
\label{eq:Langevin}
\D X_t = - \nabla V(X_t) \, \D t + \D W_t.
\end{equation}
In molecular dynamics applications, $V(x)$ represent the potential energy 
of a molecule, while in other applications
$V = - \, \halfs \log \pi + \mbox{const}$
where $\pi: \RR^m \rightarrow \RR^+$ is an invariant measure.
$V$ is usually defined on $\RR^m$, infinitely differentiable, and 
satisfies all of the assumptions in this paper so the theory is fully 
applicable, leading to order 1 strong convergence.

\subsubsection{FENE model}

The FENE (Finitely Extensible Nonlinear Elastic) model is a Langevin 
equation describing the motion of a long-chained polymer in a liquid
\cite{bs07,glw16}. The unusual feature of the FENE model is that the
potential $V(x)$ becomes infinite for finite values of $x$.
In the simplest case of a molecule with a single bond, $x$ is 
three-dimensional and $V(x)$ takes 
the form
$
V(x) = - \log (1 \!-\! \|x\|^2).
$
The SDE is defined on $\|x\|\!<\!1$, with the drift term ensure that 
$\|X_t\|\!<\! 1$ for all $t\!>\!0$.  Also, it can be verified that 
$\langle x,f(x) \rangle\!\leq\! 0$.

Because the SDE is not defined on all of $\RR^3$, the theory in 
this paper is not applicable.  However, it was one of the original
motivations for the analysis in this paper, since it seems natural 
to use an adaptive timestep, taking smaller timestep as $\|\hX_t\|$ 
approaches 1, to maintain good accuracy, as the drift varies so
rapidly near the boundary, and to greatly reduce the possibility 
of needing to clamp the computed solution to prevent it from 
crossing a numerical boundary at radius $1\!-\!\delta$ for some 
$\delta\!\ll\! 1$ \cite{glw16}.
Numerical results indicate that the order of strong convergence 
is very close to 1.

\begin{figure}
\begin{center}
\includegraphics[width=0.9\textwidth]{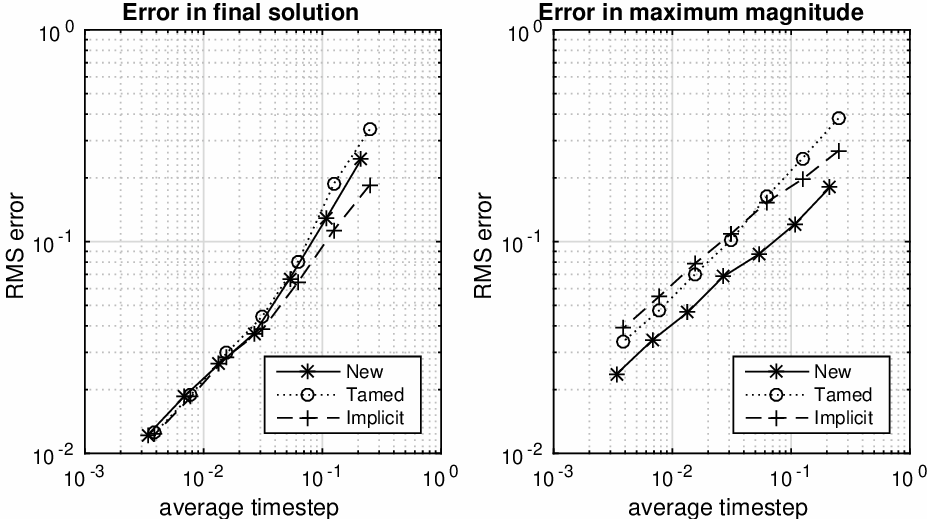}
\end{center}
\caption{Numerical results for testcase 1}
\label{fig1}
\end{figure}

\subsection{Numerical results}

The numerical tests include three testcases from \cite{hjk12} 
plus one new test which provides some motivation for the research 
in this paper.

\subsubsection{Testcase 1}

The first scalar testcase taken from \cite{hjk12} is
\[
\D X_t = - X_t^5\, \D t + X_t\, \D W_t,  ~~~~ X_0 = 1,
\]
with $T\!=\!1$.  The three methods tested are the Tamed Euler 
scheme, with $C\!=\!1$, the implicit Euler scheme, and the new
Euler scheme with adaptive timestep
\[
h^\delta(x) = \delta \, \frac{\max(1,|x|)}{\max(1,|f(x)|}.
\]
Figure \ref{fig1} shows the the root-mean-square error
plotted against
the average timestep. The plot on the left shows the error 
in the terminal time, while the plot on the right shows the 
error in the maximum magnitude of the solution.  The error 
in each case is computed by comparing the numerical solution 
to a second solution with a timestep, or $\delta$, which 
is 4 times smaller.

When looking at the error in the final solution, all 3 methods 
have similar accuracy with $\halfs$ order strong convergence.
However, as reported in \cite{hjk12}, the cost of the implicit 
method per timestep is much higher.
The plot of the error in the maximum magnitude shows that the 
new method is slightly more accurate, presumably because it 
uses smaller timesteps when the solution is large.  The plot 
was included to show that comparisons between numerical methods 
depend on the choice of accuracy measure being used.

\begin{figure}
\begin{center}
\includegraphics[width=0.9\textwidth]{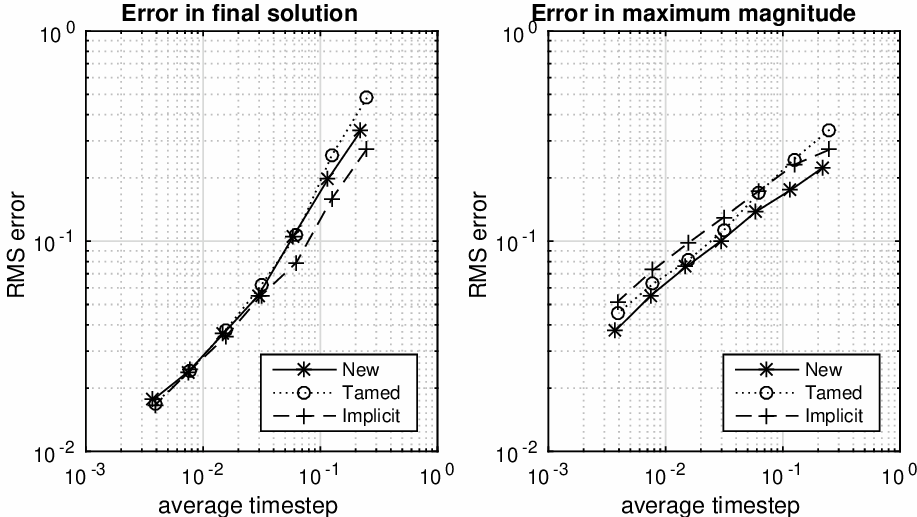}
\end{center}
\caption{Numerical results for testcase 2}
\label{fig2}
\end{figure}

\subsubsection{Testcase 2}

The second scalar testcase taken from \cite{hjk12} is
\[
\D X_t = (X_t \!-\! X_t^3)\, \D t + X_t\, \D W_t,  ~~~~ X_0 = 1,
\]
with $T\!=\!1$. The results in Figure \ref{fig2}
are similar to the first testcase.

\begin{figure}
\begin{center}
\includegraphics[width=0.9\textwidth]{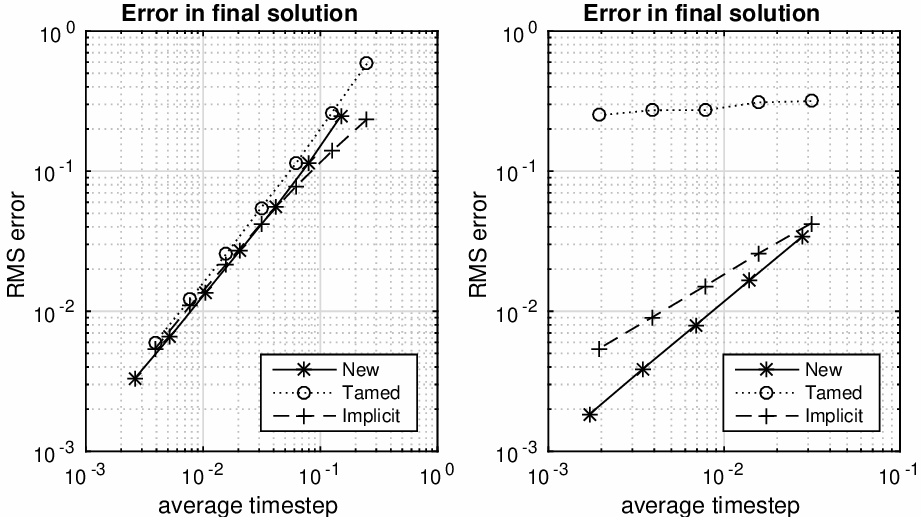}
\end{center}
\caption{Numerical results for testcases 3 (on left) and 4 (on right)}
\label{fig3}
\end{figure}

\subsubsection{Testcase 3}

The third testcase taken from \cite{hjk12} is 10-dimensional,
\[
\D X_t = (X_t \!-\! \|X_t\|^2 \, X_t)\, \D t + \D W_t,  ~~~~ X_0 = 0,
\]
with $T\!=\!1$. The results in the left-hand plot in 
Figure \ref{fig3} show that the error in the final value exhibits
order 1 strong convergence using all 3 methods, as expected.

\subsubsection{Testcase 4}

The final testcase is for the 3-dimensional FENE SDE discussed previously,
\[
\D X_t = - \frac{X_t}{1 \!-\! \|X_t\|^2} \ \D t + \D W_t, ~~~~X_0 = 0,
\]
with $T\!=\!1$.  As commented on previously, this SDE is not 
covered by the theory in this paper, but it is a motivation for 
the research because it is natural to use an adaptive timestep of
the form
\[
h^\delta(x) = \frac{\delta}{4} \, (1\!-\!\|x\|^2)
\]
to reduce the timestep when $\|\hX_t\|$ approaches the maximum radius.

All three methods are clamped so that they do not exceed a radius 
of $r_{max}=1\!-\!10^{-10}$; if the new computed value $\hX_{t_{n+1}}$ exceeds 
this radius then it is replaced by $(r_{max}/\|\hX_{t_{n+1}}\|) \hX_{t_{n+1}}$.

The numerical results in the right-hand plot in Figure \ref{fig3} 
show that the new scheme is considerably more accurate than either 
of the others, confirming that an adaptive timestep is desirable in 
this situation in which the drift varies enormously as $\|\hX_t\|$ 
approaches the maximum radius.

\newpage

\section{Proofs}

\label{sec:proofs}

This section has the proofs of the three main theorems in this paper,
one on stability, and two on the order of strong convergence.

\subsection{Theorem \ref{thm:stability}}

\begin{proof}

The proof proceeds in four steps.  First, we introduce a constant $K$ 
to modify our discretisation scheme. Second, we derive an upper bound 
for $\|\hX_t^K\|^p$. Third, we show that the moments
$\EE[ \sup_{0\leq t\leq T}\|\hX_t^K\|^p ]$ are each bounded by a constant $C_{p,T}$ 
which depends on $p$ and $T$ but is independent of $K$. Finally, we 
reach the desired conclusion by taking the limit $K\!\rightarrow\!\infty$ 
and using the Monotone Convergence theorem.

The proof is given for $p\!\geq\! 4$; the result for 
$0\!\leq\! p \!<\! 4$ follows from H{\"o}lder's inequality.

\noindent
\textbf{{Step 1: K-Scheme definition}}

For any $K\!>\!\|X_0\|,$ we modify our discretisation scheme to:
\begin{equation}
\hX_\tnp^K = P_K\left(\hX_\tn^K + f(\hX_\tn^K)\, h_n + g(\hX_\tn^K)\, \DW_n\right),
\label{Kdiscrete}
\end{equation}
where $P_K(Y)\triangleq\min(1,K/\|Y\|)\, Y$ and therefore 
$\|\hX_\tn^K\|\!\leq\! K,\, \forall n$.
The piecewise constant approximation for intermediate times is again 
$\bX_t^K=\hX_\td^K$, and the continuous approximation is
\[
\hX_t^K=P_K\left( \hX_\td^K+f(\hX_\td^K)\, (t\!-\!\td) + g(\hX_\td^K)\, (W_t\!-\!W_\td)\right).
\]
Since $h(x)$ is continuous and strictly positive, it follows that
\[
h_{min}^K \triangleq \inf_{\|x\|\leq K}h(x)>0.
\]
This strictly positive lower bound for the timesteps implies that
$T$ is attainable.


\noindent
\textbf{{Step 2: $p$th-moment of K-Scheme solution}}

$\|P_K(Y)\| \!\leq\! \|Y\|$, so if we define $\phi(x)\triangleq x\!+\!h(x)f(x)$,
then (\ref{Kdiscrete}) gives 
\begin{eqnarray*}
\|\hX_\tnp^K\|^2  & \leq & \|\hX_\tn^K\|^2 
+ 2\, h_n\left( \langle\hX_\tn^K,f(\hX_\tn^K)\rangle + \halfs h_n \|f(\hX_\tn^K)\|^2 \right)
\\ &&
+\, 2\,\langle \phi(\hX_\tn^K), g(\hX_\tn^K)\,\DW_n \rangle
+ \|g(\hX_\tn^K)\, \DW_n\|^2
\end{eqnarray*}
Using condition (\ref{eq:timestep}) for $h(x)$ then gives
\begin{eqnarray}
\label{2os3fs4}
\|\hX_\tnp^K\|^2 & \leq & 
\|\hX_\tn^K\|^2 \ +\  2\,\alpha\|\, \hX_\tn^K\|^2h_n \ +\ 2\,\beta\, h_n 
\nonumber\\&& 
+\, 2\,\langle \phi(\hX_\tn^K), g(\hX_\tn^K)\,\DW_n \rangle
+ \|g(\hX_\tn^K)\, \DW_n\|^2.
\end{eqnarray}
Similarly, for the partial timestep from $\td$ to $t$, 
since $(t\!-\!\td)\leq h_{n_t}$
\begin{equation}
\langle \hX_\td^K, f(\hX_\td^K) \rangle + \halfs \, (t\!-\!\td)\, \|f(\hX_\td^K)\|^2 
\leq \alpha \, \|\hX_\td^K\|^2 + \beta,
\label{eq:partial}
\end{equation}
and therefore we obtain
\begin{eqnarray}
\label{2os3fs4t}
\|\hX_t^K\|^2 & \leq & 
\|\hX_\td^K\|^2 \ +\  2\,\alpha\|\, \hX_\td^K\|^2 (t\!-\!\td) \ +\  2\,\beta\, (t\!-\!\td)
\nonumber\\&& 
+ \, 2\,\langle \hX_\td^K \!+\! f(\hX_\td^K)\, (t\!-\!\td), g(\hX_\td^K)\,(W_t\!-\!W_\td) \rangle
\nonumber\\&& 
+ \, \|g(\hX_\td^K)\, (W_t\!-\!W_\td)\|^2.
\end{eqnarray}
Summing (\ref{2os3fs4}) over multiple timesteps and then adding (\ref{2os3fs4t}) gives
\begin{eqnarray*}
\|\hX^K_t\|^2 & \leq & 
\|X_0\|^2 \ +\ 2\,\alpha\left( \sum_{k=0}^{n_t-1}\|\hX^K_{t_k}\|^2h_k + \|\hX_\td^K\|^2 (t\!-\!\td) \right)
\ +\ 2\, \beta\, t
\\[-0.1in] &&
+\, 2 \sum_{k=0}^{n_t-1}\langle \phi(\hX_{t_k}^K), g(\hX_{t_k}^K)\DW_k \rangle)
+  \sum_{k=0}^{n_t-1}\|g(\hX_{t_k}^K)\, \DW_k\|^2
\\ &&
+\,  2 \langle  \hX_\td^K \!+\! f(\hX_\td^K)\, (t\!-\!\td),\, g(\hX_\td^K)(W_t\!-\!W_\td) \rangle
\\[0.05in] &&
+\, \|g(\hX_\td^K)\, (W_t\!-\!W_\td)\|^2.
\end{eqnarray*}

Re-writing the first summation as a Riemann integral, and the second 
as an It{\^o} integral, raising both sides to the power $p/2$ 
and using Jensen's inequality, we obtain
\begin{eqnarray}
\label{os3fs4}
\|\hX^K_t\|^p & \!\!\leq\!\! & 
7^{p/2 - 1}\left\{ \rule{0in}{0.25in}
\|X_0\|^p \ +\  
\left(2\, \alpha \int_0^t\|\bX_s^K\|^2\,\D s\right)^{p/2}\ +\  
(2\, \beta\, t)^{p/2} 
\right. \nonumber \\ && 
~~~~ + \, \left|\, 2\! \int_0^\td \! \langle \phi(\bX_s^K), g(\bX_s^K)\, \D W_s\rangle \,\right|^{p/2}
\!\! + \left(\sum_{k=0}^{n_t-1}\|g(\bX_{t_k}^K) \, \DW_k \|^2 \! \right)^{p/2}
\nonumber \\ &&
~~~~~~~~ + \, \left| 2\, \langle  \bX_t^K \!+\! f(\bX_t^K)\, (t\!-\!\td), g(\bX_t^K) (W_t\!-\!W_\td) \rangle \right|^{p/2}
\!\!
\nonumber \\ && \left.
~~~~~~~~ +\, \|g(\bX_t^K) (W_t\!-\!W_\td)\|^p 
\rule{0in}{0.25in} \right\}.
\end{eqnarray}


\noindent
\textbf{{Step 3: Expected supremum of $p$th-moment of K-Scheme}}

For any $0\leq t\leq T$ we take the supremum on both sides of inequality (\ref{os3fs4}) and 
then take the expectation to obtain
\[
\EE\left[ \sup_{0\leq s\leq t}\|\hX^K_s\|^p\right] \ \leq \ 
 7^{p/2 - 1}\left( I_1 + I_2  + I_3  + I_4 + I_5 \right), 
\]
where
\begin{eqnarray*}
I_1 &=& \|X_0\|^p + \EE\left[ \left(2 \alpha \int_0^t\|\bX_s^K\|^2\,\D s\right)^{p/2} \right] + (2\, \beta\, t)^{p/2}, 
\\[0.05in]
I_2 &=& \EE\left[ \sup_{0\leq s\leq \td}\left| \,2\! \int_0^s 
\langle \phi(\bX_u^K), g(\bX_u^K)\,\D W_u \rangle \right|^{p/2} \right],
\\[0.05in]
I_3 &=& \EE\left[ \left(\sum_{k=0}^{n_t-1}\|g(\bX_{t_k}^K) \, \DW_k\|^2\right)^{p/2}\right],
\\[0.05in]
I_4 &=& \EE\left[ \sup_{0\leq s\leq t} 
\left|2 \langle \bX_s^K \!+\! f(\bX_s^K)\, (s\!-\!\sd), g(\bX_s^K)  (W_s\!-\!W_\sd) \rangle \right|^{p/2}\right],
\\[0.05in]
I_5 &=& \EE\left[\sup_{0\leq s\leq t} \|g(\bX_s^K) \, (W_s\!-\!W_\sd)\|^p \right].
\end{eqnarray*}
We now consider $I_1, I_2, I_3, I_4, I_5$ in turn.  Using Jensen's inequality, we obtain
\[
I_1 \leq \|X_0\|^p 
+ (2 \alpha)^{p/2} T^{p/2-1}\!\! \int_0^t\EE\left[ \sup_{0\leq u \leq s} \|\hX_u^K\|^p\right]\,\D s
+ (2\, \beta\, T)^{p/2}.
\]

For $I_2$, we begin by noting that ue to condition (\ref{eq:timestep}), 
for $u\!<\!\td$ we have
\begin{eqnarray*}
\|\phi(\bX_u^K)\|^2 
   &=& \|\bX_u^K\|^2 + 2\, h(\bX_u^K)\, 
\left( \langle \bX_u^K, f(\bX_u^K)\rangle + \halfs\, h(\bX_u^K) \|f(\bX_u^K)\|^2 \rule{0in}{0.16in} \right)
\\ &\leq&
\|\bX_u^K\|^2+ 2 \, h(\bX_u^K) \, (\alpha\|\bX_u^K\|^2 +\beta)
\\ &\leq& (1+2\,\alpha\, T) \|\bX_u^K\|^2 + 2\, \beta\, T,
\end{eqnarray*}
and hence by Jensen's inequality
\[
\|\phi(\bX_u^K)\|^{p/2} \leq 2^{p/4-1} \left(\rule{0in}{0.16in}
(1+2\,\alpha\, T)^{p/4} \|\bX_u^K\|^{p/2} + (2\,\beta\, T)^{p/4} 
\right).
\]
In addition, the linear growth condition (\ref{eq:g_growth}) gives
\[
\|g(\bX_u^K)\|^{p/2}
\leq  2^{p/4-1} \left( \alpha^{p/4} \|\bX_u^K\|^{p/2} + \beta^{p/4} \right),
\]
and combining the last two equation, there exists a constant 
$c_{p,T}$ depending on $p$ and $T$,
in addition to $\alpha, \beta$, such that
\[
\|\phi(\bX_u^K)^T g(\bX_u^K) \|^{p/2} \,\leq\, c_{p,T} 
\left( \|\bX_u^K\|^p + 1\right).
\]
Then, by the Burkholder-Davis-Gundy inequality, there is a constant $C_p$ such that
\begin{eqnarray*}
I_2 &\leq & 
C_p\, 2^{p/2}\, \EE\left[ \left(\int_0^t \| \phi(\bX_u^K)^T g(\bX_u^K) \|^2 \, \D u \right)^{p/4} \right] 
\\ &\leq&
C_p\, 2^{p/2} \, T^{p/4-1} \, \EE\left[ \int_0^t \| \phi(\bX_u^K)^T g(\bX_u^K) \|^{p/2} \, \D u \right] 
\\ &\leq& 
c_{p,T}\, C_p\, 2^{p/2} \, T^{p/4-1} \, \left(
\int_0^t \EE\left[ \sup_{0\leq u\leq s}\|\hX_u^K\|^p\right] \D s \ +\ T
\right).
\label{os3fn8}
\end{eqnarray*}

For $I_3$, we start by observing that by standard results there 
exists a constant $c_p$ which depends solely on $p$ such that 
for any $t_k\!\leq\!s<t_{k+1}$,
\begin{equation}
\EE[ \sup_{t_k\leq u \leq s} \|\, W_u\!-\!W_{t_k} \,\|^p\ |\ \mathcal{F}_{t_k}] 
\ =\ c_p\, (s\!-\!\sd)^{p/2}.
\label{eq:DW}
\end{equation}
One variant of Jensen's inequality, 
when $h_k, u_k$ are both positive and $p\!\geq\! 1$, is
\[
\left( \sum_k h_k\, u_k \right)^p  \leq
\left( \sum_k h_k \right)^{p-1} \sum_k h_k u^p_k.
\]
Using this, and (\ref{eq:DW}) with $s\equiv t_{k+1}$ so that $s-\sd = h_k$,
\begin{eqnarray*}
I_3 &\leq& 
T^{p/2-1}\, 
\EE\left[\ \sum_{k=0}^{n_t-1} h_k \, \|g(\bX_{t_k}^K)\|^p\ \frac{\|\DW_k\|^p}{h_k^{p/2}}
\right]
\\ &\leq & T^{p/2-1}\, c_p\ \EE\left[\ \int_0^t \|g(\bX_s^K)\|^p\, \D s \right].
\end{eqnarray*}
Using condition (\ref{eq:g_growth}), and Jensen's inequality, we then obtain
\[
I_3 \leq (2\,T)^{p/2-1}\, c_p \left(
\alpha^{p/2} \int_0^t \EE\left[ \sup_{0\leq u \leq s} \|\hX_u^K\|^p\right]\D s
\ +\ \beta^{p/2}\, T\right).
\]

For $I_4$, using (\ref{eq:partial}) and following the same argument as 
for $I_2$, there exists a constant $c_{p,T}$ depending on both $p$ and 
$T$ such that
\[
\|\bX_s^K\!+\! f(\bX_s^K) (s\!-\!\sd)\|^{p/2} \| g(\bX_s^K)\|^{p/2} \leq
c_{p,T}\, \left( \|\bX_s^K\|^p + 1\right).
\]
Therefore, again using (\ref{eq:DW}),
\begin{eqnarray*}
I_4 &\leq& 
2^{p/2}\, \EE\left[ \sup_{0\leq s\leq t} \left| 
\langle \bX_s^K\!+\! f(\bX_s^K) (s\!-\!\sd), g(\bX_s^K)\, (W_s\!-\!W_\sd) \rangle \right|^{p/2}\right]
\\ &\leq&
 c_{p,T} \, 2^{p/2}\, \EE\left[ \sum_{k=0}^{n_t-1} 
\left(\|\bX_{t_k}^K\|^p \!+\! 1\right) \sup_{t_k\leq s< t_{k+1}} \| (W_s\!-\!W_\sd)\|^{p/2} \right.
\\&& \hspace{1.0in} \left. +\, 
\left(\|\bX_t^K\|^p \!+\! 1\right) \sup_{\td\leq s\leq t} \| (W_s\!-\!W_\sd)\|^{p/2}\right]
\\ &\leq& 
c_{p/2}  \, c_{p,T}\, 2^{p/2}\, T^{p/4-1}\ 
\EE\left[ \sum_{k=0}^{n_t-1} \!\! \left(\|\bX_{t_k}^K\|^p \!+\! 1\right)\! h_k
                             + \left(\|\bX_t^K\|^p \!+\! 1\right)\!(t\!-\!\td) \right] 
\\ &\leq& c_{p/2} \, c_{p,T} \, 2^{p/2}\, T^{p/4-1} \left(
\int_0^t  \EE\left[ \sup_{0\leq u \leq s} \|\hX_u^K\|^p\right] \D s \ +\ T
\right).
\end{eqnarray*}

Similarly, using the same definition for $c_p$, we have
\begin{eqnarray*}
I_5 &\leq& 
c_p \, (2\,T)^{p/2-1} \left(
\alpha^{p/2} \int_0^t   \EE\left[ \sup_{0\leq u \leq s} \|\hX_u^K\|^p \right] \D s
\ +\ \beta^{p/2}\, T
\right).
\end{eqnarray*}

Collecting together the bounds for $I_1, I_2, I_3, I_4, I_5$,
we conclude that there exist constants $C^1_{p,T}$ and $C^2_{p,T}$ such that
\[
\EE\left[ \sup_{0\leq s\leq t}\|\hX^K_{s}\|^p\right]
\leq C_{p,T}^1+C_{p,T}^2 \int_0^t\EE\left[ \sup_{0\leq u\leq s}\|\hX_u^K\|^p\right] \D s,
\]
and Gr\"{o}nwall's inequality gives the result
\[
\EE\left[ \sup_{0\leq t\leq T}\|\hX^K_t\|^p\right] 
\,\leq\, C_{p,T}^{1} \, \exp(C_{p,T}^{2}\, T) 
\,\triangleq\, C_{p,T} 
\,<\, \infty.
\]

\noindent
\textbf{{Step 4: Expected supremum of $p$th-moment of $\hX_t$}}

For any $\omega\!\in\!\Omega$, $\hX_t\!=\!\hX_t^K$ 
for all $0\!\leq\!t\!\leq\!T$ if, and only if, 
$\sup_{0\leq t\leq T} \|\hX_t\|
\!\leq\! K$.
Therefore, by the Markov inequality,
\[
  \PP(\sup_{0\leq t\leq T} \|\hX_t\| < K)
\, =\, \PP(\sup_{0\leq t\leq T} \|\hX_t^K\| < K)
\,\geq\, 1 \!-\! \EE[\sup_{0\leq t\leq T} \|\hX_t^K\|^4] / K^4 
\,\rightarrow\, 1
\]
as $K\rightarrow\infty$. Hence, almost surely, 
$\displaystyle \sup_{0\leq t\leq T} \|\hX_t\| < \infty$
and $T$ is attainable.  Also,
\[
\lim_{K\rightarrow\infty} \sup_{0\leq t\leq T}\|\hX^K_t(\omega)\|
= \sup_{0\leq t\leq T}\|\hX_t(\omega)\|
\]
and for $0\!<\! K_1 \!\leq\! K_2$,
\[
\sup_{0\leq t\leq T}\|\hX^{K_1}_t(\omega)\|
\ \leq\ \sup_{0\leq t\leq T}\|\hX^{K_2}_t(\omega)\|
\ \leq\ \sup_{0\leq t\leq T}\|\hX_t(\omega)\|.
\]
Therefore, by the Monotone Convergence Theorem, 
\[
\EE\left[ \sup_{0\leq t\leq T}\|\hX_t\|^p\right]
\ =\ \lim_{K\rightarrow\infty} \EE\left[ \sup_{0\leq t\leq T}\|\hX^K_t\|^p\right]
\ \leq\ C_{p,T}.
\]
\end{proof}

\subsection{Theorem \ref{thm:convergence_order}}

\begin{proof}
The approach which is followed is to bound the approximation
error $e_t\triangleq \hX_t - X_t$ by terms which depend on either 
$\hX_s\!-\!\bX_s$ or $e_s$, and then use local analysis within 
each timestep to bound the former, and Gr{\"o}nwall's inequality 
to handle the latter.

The proof is again given for $p\!\geq\! 4$; the result for 
$0\!\leq\! p \!<\! 4$ follows from H{\"o}lder's inequality.

We start by combining the original SDE with (\ref{eq:approx_SDE})
to obtain
\[
\D e_t
\ =\ \left( f(\bX_t)\!-\!f(X_t)\right) \D t
   + \left( g(\bX_t)\!-\!g(X_t)\right) \D W_t,
\]
and then by It{\^o}'s formula, together with $e_0\!=\!0$, we get
\begin{eqnarray*}
\|e_t\|^2 &\leq&
2 \int_0^t \langle e_s, f(\hX_s)\!-\!f(X_s)\rangle \,\D s -
2 \int_0^t\langle e_s, f(\hX_s)\!-\!f(\bX_s)\rangle \,\D s\\ 
&&+\, \int_0^t \|g(\bX_s)\!-\!g(X_s)\|^2 \,\D s
+ 2 \int_0^t \langle e_s, (g(\bX_s)\!-\!g(X_s))\,\D W_s \rangle.
\end{eqnarray*}
Using the conditions in Assumption \ref{assp:Lipschitz}, 
(\ref{eq:onesided_Lipschitz}) implies that
\[
\langle e_s, f(\hX_s)\!-\!f(X_s)\rangle \leq \halfs \alpha\, \|e_s\|^2,
\]
(\ref{eq:local_Lipschitz}) implies that
\begin{eqnarray*}
\left| \langle e_s, f(\hX_s)\!-\!f(\bX_s) \rangle \right|
  &\leq& \|e_s\| \, L(\hX_s,\bX_s) \, \| \hX_s\!-\!\bX_s \|
\\&\leq& \halfs \|e_s\|^2 + \halfs L(\hX_s,\bX_s)^2 \| \hX_s\!-\!\bX_s \|^2
\end{eqnarray*}
where $L(x,y)\triangleq \gamma(\|x\|^q + \|y\|^q) + \mu$, 
and (\ref{eq:g_Lipschitz}) gives
\[
\|g(\bX_s)\!-\!g(X_s)\|^2
\ \leq\ \halfs\, \alpha \,  \|\bX_s\!-\!X_s\|^2
\ \leq\ \alpha\, \|e_s\|^2 + \alpha \, \|\hX_s\!-\!\bX_s\|^2.
\]
Hence,
\begin{eqnarray*}
  \|e_t\|^2 &\leq&  (2\alpha\!+\!1) \int_0^t \| e_s \|^2 \,\D s 
+ \int_0^t \left( L(\hX_s,\bX_s)^2 \!+\! \alpha \right)\|\hX_s\!-\!\bX_s\|^2\, \D s
\\&& +\, 2 \int_0^t \langle e_s, (g(\bX_s)\!-\!g(X_s))\,\D W_s \rangle.
\end{eqnarray*}
and then by Jensen's inequality we obtain
\begin{eqnarray*}
\|e_t\|^p &\leq&
(3T)^{p/2-1} (2\alpha\!+\!1)^{p/2} \int_0^t \| e_s \|^p \,\D s 
\\ &+& (3T)^{p/2-1} \int_0^t 
\left( L(\hX_s,\bX_s)^2 \!+\!  \alpha \right)^{p/2}\|\hX_s\!-\!\bX_s\|^p\, \D s
\\ &+& 3^{p/2-1} 2^{p/2}
\left| \int_0^t \langle e_s, (g(\bX_s)\!-\!g(X_s))\,\D W_s \rangle \right|^{p/2}.
\end{eqnarray*}

Taking the supremum of each side, and then the expectation yields
\begin{eqnarray*}
\EE\left[\sup_{0\leq s \leq t} \|e_s\|^p \right]
&\leq&
  (3T)^{p/2-1}  (2\alpha\!+\!1)^{p/2} \int_0^t \EE\left[\sup_{0\leq u\leq s} \| e_u \|^p \right] \,\D s 
\\ &+& (3T)^{p/2-1} \!\! \int_0^t 
\! \EE\!\left[ \left(\! L(\hX_s,\bX_s)^2 \!+\! \alpha \right)^{p/2} \! \|\hX_s\!-\!\bX_s\|^p\right] \D s
\\ &+& 
 3^{p/2-1} 2^{p/2} \EE\left[
\sup_{0\leq s \leq t} \left| \int_0^s \langle e_u, (g(\bX_u)\!-\!g(X_u))\,\D W_u \rangle \right|^{p/2}
\right].
\end{eqnarray*}
By the H{\"o}lder inequality, 
\begin{eqnarray*}
\lefteqn{\hspace*{-0.5in}
\EE\left[ \left( L(\hX_s,\bX_s)^2 \!+\! \alpha \right)^{p/2} \|\hX_s\!-\!\bX_s\|^p\right] 
}
\\ &\leq&
\left( \EE\left[ \left( L(\hX_s,\bX_s)^2 \!+\! \alpha \right)^p\right]
 \EE\left[ \|\hX_s\!-\!\bX_s\|^{2p}\right]\right)^{1/2},
\end{eqnarray*}
and $\EE\left[ \left( L(\hX_s,\bX_s)^2 \!+\! \alpha \right)^p \right]$
is uniformly bounded on $[0,T]$ 
due to the stability property in Theorem \ref{thm:stability}.

In addition, by the Burkholder-Davis-Gundy inequality (which gives the 
constant $C_p$ which depends only on $p$) followed by Jensen's inequality 
plus the Lipschitz condition for $g$, we obtain
\begin{eqnarray*}
\lefteqn{\hspace*{-0.2in}
\EE\left[
\sup_{0\leq s \leq t} \left| \int_0^s \langle e_u, (g(\bX_u)\!-\!g(X_u))\,\D W_u \rangle \right|^{p/2}
\right]
} 
\\ &\leq&
C_p\ \EE\left[ \left( \int_0^t \|e_s\|^2 \|g(\bX_s)\!-\!g(X_s)\|^2 \, \D s \right)^{p/4}\right].
\\ &\leq&
C_p \ T^{p/4-1}\, (\halfs \alpha)^{p/4} \EE\left[ \int_0^t \|e_s\|^{p/2} \|\bX_s\!-\!X_s\|^{p/2} \, \D s \right]
\\ &\leq&
C_p \ T^{p/4-1}\, (\halfs \alpha)^{p/4} \EE\left[ \int_0^t \halfs \, \|e_s\|^p + \halfs \, \|\bX_s\!-\!X_s\|^p \, \D s \right]
\\ &\leq&
C_p \ T^{p/4-1}\, (\halfs \alpha)^{p/4} \EE\left[ \int_0^t (\halfs \!+\!2^{p-2})\|e_s\|^p + 2^{p-2} \|\hX_s\!-\!\bX_s\|^p \, \D s \right].
\end{eqnarray*}
Hence, using $\EE[ \|\hX_s\!-\!\bX_s\|^p] \leq (\EE[ \|\hX_s\!-\!\bX_s\|^{2p}])^{1/2}$, 
there are constants $C^1_{p,T}, C^2_{p,T}$ such that
\begin{equation}
\EE\left[\sup_{0\leq s \leq t} \|e_s\|^p \right]
\leq
C^1_{p,T}\! \int_0^t \!\EE\!\left[\sup_{0\leq u \leq s} \|e_u\|^p \right] \D s
\, +\,
C^2_{p,T}\! \int_0^t \!\left( \EE\!\left[ \|\hX_s\!-\!\bX_s\|^{2p} \right] \right)^{\!1/2} 
\!\!\D s.
\label{eq:halfway}
\end{equation}


For any $s\!\in\![0,T]$,
$\hX_s\!-\!\bX_s = f(\hX_{\sd}) (s\!-\!\sd) + g(\hX_{\sd}) (W_s \!-\! W_\sd)$,
and hence, by a combination of Jensen and H{\"o}lder inequalities, 
we get
\begin{eqnarray*}
\EE\left[ \|\hX_s\!-\!\bX_s\|^{2p}\right] 
&\leq& 2^{2p-1} \left( \EE\left[ \|f(\hX_{\sd})\|^{4p} \right]
\EE \left[ (s\!-\!\sd)^{4p} \right]\right)^{1/2}
\\ &+& 
2^{2p-1}  \left( \EE\left[ \|g(\hX_{\sd})\|^{4p} \right]
                 \EE\left[ \|W_s \!-\! W_\sd\|^{4p} \right] \right)^{1/2}.
\end{eqnarray*}
$\EE[ \|f(\hX_{\sd})\|^{4p}]$ and $\EE[ \|g(\hX_{\sd})\|^{4p}]$ are 
both uniformly bounded on $[0,T]$ due to stability and the polynomial 
bounds on the growth of $f$ and $g$. Furthermore, we have 
$\EE[ (s\!-\!\sd)^{4p}]
\leq (\delta T)^{4p} \leq \delta^{2p} T^{4p}$,
and by standard results there is a constant $c_p$ such that
$\EE[ \|W_s \!-\! W_\sd\|^{4p}] 
= \EE[\ \EE[\|W_s \!-\! W_\sd\|^{4p}\, |\, {\cal F}_\sd] \ ]
\leq c_p (\delta T)^{2p}$.
Hence, there exists a constant $C^3_{p,T}\!>\!0$ such that
$\EE[\, \|\hX_s\!-\!\bX_s\|^{2p}] \leq C^3_{p,T}\, \delta^p$,
and therefore equation (\ref{eq:halfway}) gives us
\[
\EE\left[\sup_{0\leq s \leq t} \|e_s\|^p \right]
\leq
C^1_{p,T}\! \int_0^t \!\EE\!\left[\sup_{0\leq u \leq s} \|e_u\|^p \right] \D s
\, +\,
C^2_{p,T} \sqrt{C^3_{p,T}}\ T \,  \delta^{p/2},
\]
and Gr\"{o}nwall's inequality then provides the final result.
\end{proof}


\subsection{Theorem \ref{thm:convergence_order2}}

\begin{proof}
The error $e_t\triangleq \hX_t - X_t$ satisfies the SDE
$\D e_t = \left( f(\bX_t)\!-\!f(X_t) \right)\,\D t$
and hence
\begin{eqnarray*}
\|e_t\|^2
&=&
2 \int_0^t\langle e_s, f(\hX_s)\!-\!f(X_s)\rangle \,\D s \ -
2 \int_0^t\langle e_s, f(\hX_s)\!-\!f(\bX_s) \rangle \,\D s
\nonumber \\&\leq &
\alpha \int_0^t\|e_s\|^2\,\D s
- 2 \int_0^t \langle e_s, f(\hX_s)\!-\!f(\bX_s)\rangle \,\D s,
\end{eqnarray*}
due to the one-sided Lipschitz condition (\ref{eq:onesided_Lipschitz}), 
so therefore
\begin{eqnarray}
\EE\left[\sup_{0\leq s\leq t}\|e_s\|^p \right]
&\leq&\ 
\alpha^{p/2} (2T)^{p/2-1}\int_0^t\EE\left[\sup_{0\leq u\leq s}\|e_u\|^p \right]\,\D s
\nonumber \\ \label{eq:thm4}
&& +\ 
2^{p-1}\, \EE\left[\sup_{0\leq s\leq t}\left| 
\int_0^s \langle e_u, f(\hX_u)\!-\!f(\bX_u) \rangle \,\D u\,
\right|^{p/2}\right].  
\end{eqnarray}
Within a single timestep, 
$\hX_s\!-\!\bX_s = f(\bX_s)(s\!-\!\sd) + (W_s\!-\!W_\sd)$,
and therefore Lemma \ref{lemma:useful} gives
\begin{eqnarray*}
\langle e_s, f(\hX_s)\!-\!f(\bX_s) \rangle 
&=& \langle e_s, \nabla f(\bX_s) (\hX_s\!-\!\bX_s) \rangle\ +\ R_s \\[0.1in]
&=& \langle e_s, (s\!-\!\sd) \nabla f(\bX_s) f(\bX_s) \rangle\ +\ R_s \\
&&  +\ \langle (e_s\!-\!e_\sd),  \nabla f(\bX_s) (W_s\!-\!W_\sd) \rangle \\
&&  +\ \langle e_\sd, \nabla f(\bX_s) (W_s\!-\!W_\sd) \rangle,
\end{eqnarray*}
where 
$|R_s|\leq \left( \gamma\, (\|\hX_s\|^q \!+\! \|\bX_s\|^q) + \mu\right)\, \|e_s\| \, \|\hX_s\!-\!\bX_s \|^2$.
Hence,
\[
\EE\left[\sup_{0\leq s\leq t}\left| 
\int_0^s \langle e_u, f(\hX_u)\!-\!f(\bX_u) \rangle \,\D u \,
\right|^{p/2}\right]  \leq 4^{p/2-1} (I_1 + I_2 + I_3 + I_4),
\]
where
\begin{eqnarray*}
I_1 & = & \EE\left[\sup_{0\leq s\leq t}\left|\int_0^s
 \langle e_u, (u\!-\!\ud)  \nabla f(\bX_u) f(\bX_u) \rangle
 \, \D u \,\right|^{p/2}\right], \\
I_2 & = & \EE\left[\sup_{0\leq s\leq t}\left|\int_0^s R_u
 \ \D u \,\right|^{p/2}\right], \\
I_3 & = & \EE\left[\sup_{0\leq s\leq t}\left|\int_0^s
\langle (e_u\!-\!e_\ud), \nabla f(\bX_u) (W_u\!-\!W_\ud) \rangle
 \, \D u \,\right|^{p/2}\right],  \\
I_4 & = & \EE\left[\sup_{0\leq s\leq t}\left|\int_0^s
\langle e_\ud, \nabla f(\bX_u) (W_u\!-\!W_\ud) \rangle
 \, \D u \,\right|^{p/2}\right].
\end{eqnarray*}

We now bound $I_1, I_2, I_3, I_4$ in turn.  Noting that $s\!-\!\sd\leq \delta\, T$,
\begin{eqnarray*}
I_1 &\leq& 
T^{p/2-1}\int_0^t \EE\left[ \|e_s\|^{p/2} (\delta T)^{p/2} \| f(\bX_s)\|^{p/2} \|\nabla f(\bX_s)\|^{p/2}\right] \ \D s 
\\ &\leq& \halfs\, T^{p/2-1} \int_0^t \EE\left[\sup_{0\leq u \leq s} \|e_u\|^p \right]\ \D s
\\ && +\ \halfs\, T^{p/2-1} (\delta T)^p \int_0^t \EE\left[
\| f(\bX_s)\|^p \|\nabla f(\bX_s)\|^p\right] \ \D s.
\end{eqnarray*}
The last integral is finite because of stability and the polynomial bounds on the growth
of both $f$ and $\nabla f$, and hence there is a constant $C^1_{p,T}$ such that
\[
I_1 \leq  \halfs\, T^{p/2-1} \int_0^t \EE\left[\sup_{0\leq u \leq s} \|e_u\|^p \right]\ \D s + C^1_{p,T}\, \delta^p.
\]

Similarly, using the H{\"o}lder inequality,
\begin{eqnarray*}
I_2 &\!\leq\!& 
T^{p/2-1}\int_0^t \EE\left[ \|e_s\|^{p/2}\, \left( \gamma\, (\|\hX_s\|^q \!+\! \|\bX_s\|^q) + \mu\right)^{p/2}\, 
\|\hX_s\!-\!\bX_s \|^p\, \right] \D s
\\&\!\leq\!& 
\halfs \, T^{p/2-1}  \int_0^t \EE\left[\sup_{0\leq u \leq s} \|e_u\|^p \right]\ \D s
\\&&\!\!\!\!\!\! + \ \halfs \, T^{p/2-1} 
\int_0^t \left( \EE\left[ \left( \gamma\, (\|\hX_s\|^q \!+\! \|\bX_s\|^q) + \mu\right)^{2p}\right]
\, \EE\left[ \|\hX_s\!-\!\bX_s \|^{4p}\right]\right)^{1/2}\! \D s,
\end{eqnarray*}
and hence, using stability and bounds on $\EE\left[ \|\hX_s\!-\!\bX_s \|^{4p} \right]$ 
from the proof of Theorem \ref{thm:convergence_order}, there is a constant $C^2_{p,T}$ such that
\[
I_2 \leq  \halfs\, T^{p/2-1} \int_0^t \EE\left[\sup_{0\leq u \leq s} \|e_u\|^p \right]\ \D s + C^2_{p,T}\, \delta^p.
\]

For the next term, $I_3$, we start by bounding $\|e_s\!-\!e_\sd\|$.  Since
\[
e_s\!-\!e_\sd = \int_\sd^s \left( f(\bX_u) - f(X_u)\right) \, \D u,
\]
by Jensen's inequality and Assumption \ref{assp:enhanced_Lipschitz} it follows that
\begin{eqnarray*}
\| e_s\!-\!e_\sd \|^p 
   &\leq& (\delta \, T)^{p-1} \int_\sd^s \| f(\bX_u) \!-\! f(X_u) \|^p \, \D u
\\ &\leq& (2\,\delta \, T)^{p-1}
\int_\sd^s L^p(\bX_u,X_u) \left( \|e_u\|^p + \| \hX_u\!-\!\bX_u\|^p \right) \, \D u,
\end{eqnarray*}
where $L(\bX_u,X_u)\equiv \gamma(\|\bX_u\|^k\!+\!\|X_u\|^k) + \mu$.
We again have an $O(\delta^{p/2})$ bound for $\EE[ \|\hX_s\!-\!\bX_s\|^p]$,
while Theorem \ref{thm:convergence_order} proves that there is a 
constant $c_{p,T}$ such that
\[
\EE[ \| e_s \|^p] \leq c_{p,T} \, \delta^{p/2}.
\]
Combining these, and using the H{\"o}lder inequality and the finite 
bound for $\EE[L^p(\bX_u,X_u)]$ for all $p\!\geq\! 2$, due to the 
usual stability results, we find that there is a different constant 
$c_{p,T}$ such that
\[
\EE[ \| e_s\!-\!e_\sd \|^p ] \leq c_{p,T} \, \delta^{3p/2}.
\]
Now,
\begin{eqnarray*}
I_3 \leq T^{p/2-1} \int_0^t \EE\left[ 
\| e_s\!-\!e_\sd\|^{p/2}\|\nabla f(\bX_s)\|^{p/2} \|W_s\!-\!W_\sd\|^{p/2} \right] \D s,
\end{eqnarray*}
so using the H{\"o}lder inequality and the usual stability bounds, 
we conclude that there is a constant $C^3_{p,T}$ such that
\[
I_3 \leq  C^3_{p,T}\, \delta^p.
\]

Lastly, we consider $I_4$.  For the timestep $[t_n,t_{n+1}]$, 
we have 
\[
\D \left( (t\!-\!t_{n+1}) (W_t\!-\!W_{t_n})\right) = 
(W_t\!-\!W_{t_n}) \, \D t + (t\!-\!t_{n+1})\, \D W_t
\]
and therefore, integrating by parts within each timestep,
\begin{eqnarray*}
\lefteqn{ \hspace{-0.25in}
\int_0^s \langle e_\ud, \nabla f(\bX_u) (W_u\!-\!W_\ud) \rangle \, \D u
}
\\ & = &
\int_0^s (\uo \!-\!u) \, \langle e_\ud, \nabla f(\bX_u)\, \D W_u \rangle 
\ -\ (\so \!-\!s) \langle e_\sd, \nabla f(\bX_s) (W_s\!-\!W_\sd) \rangle
\end{eqnarray*}
where $\uo = \min\{t_n: t_n\!>\!u\} = t_{n_u+1}$.
Hence,
$
I_4 \leq 2^{p/2 - 1} (I_{41}+I_{42})
$
where
\begin{eqnarray*}
I_{41} &=& \EE\left[\sup_{0\leq s\leq t} \left|
\int_0^s  (\uo \!-\!u)\, \langle e_\ud, \nabla f(\bX_u)\, \D W_u \rangle 
\right|^{p/2} \right], \\
I_{42} &=& \EE\left[\sup_{0\leq s\leq t} \left|
(\so \!-\!s)\, \langle e_\sd, \nabla f(\bX_s) (W_s\!-\!W_\sd) \rangle
\right|^{p/2} \right].
\end{eqnarray*}

By the Burkholder-Davis-Gundy inequality,
\begin{eqnarray*}
I_{41} & \leq &
C_p\, \EE\left[
\left( \int_0^t (\so \!-\!s)^2 \|e_\sd\|^2 \|\nabla f(\bX_s)\|^2\, \D s
\right)^{p/4} \right] 
\\ & \leq &
C_p\, T^{3p/4-1}\, 
\EE\left[
\int_0^t \|e_\sd\|^{p/2}\, \delta^{p/2}\, \|\nabla f(\bX_s)\|^{p/2}\, \D s
\right]
\\ & \leq &
\fracs{1}{2} C_p\, T^{3p/4-1}\, 
\EE\left[
\int_0^t \left( \sup_{0\leq u\leq s}\|e_u\|^p + \delta^p\, \|\nabla f(\bX_s)\|^p
 \right) \D s
\right]
\end{eqnarray*}
with $\EE[\|\nabla f(\bX_s)\|^p]$ uniformly bounded on $[0,T]$ so that
there is a constant $C^{41}_{p,T}$ such that
\[
I_{41} \leq \fracs{1}{2}\, C_p\, T^{3p/4-1} 
\int_0^t  \EE\left[ \sup_{0\leq u\leq s}\|e_u\|^p \right] \D s
+ C^{41}_{p,T}\, \delta^p.
\]
Turning to $I_{42}$, Young's inequality and H{\"o}lder's inequality 
give
\[
I_{42} \leq \frac{1}{2\xi} \EE\left[ \sup_{0\leq s\leq t} \|e_s\|^p \right]
+ \frac{\xi}{2} (2\delta T)^p\! \left( \EE\!\left[ 
\sup_{0\leq s\leq t} \|\nabla f\|^{2p} \right]
\EE\!\left[  \sup_{0\leq s\leq t} \|W_s\|^{2p} \right] \right)^{1/2}
\]
for any $\xi\!>\!0$, 
and hence there is a constant $C^{42}_{p,T}$ such that
\[
I_{42} \leq \frac{1}{2\xi}\, \EE\left[ \sup_{0\leq s\leq t} \|e_s\|^p \right]
+ \xi\, C^{42}_{p,T}\, \delta^p.
\]

Returning to (\ref{eq:thm4}), and inserting the bounds for
$I_1$, $I_2$, $I_3$, $I_4$, $I_{41}$, and $I_{42}$, with $\xi = 2^{5p/2-4}$,
gives
\[
\EE\!\left[\sup_{0\leq s\leq t}\|e_s\|^p \right] \leq
  \fracs{1}{2} \EE\!\left[\sup_{0\leq s\leq t}\|e_s\|^p \right]
+ C^5_{p,T}\! \int_0^t \EE\!\left[\sup_{0\leq u\leq s}\|e_u\|^p \right] \D s
+ C^6_{p,T} \, \delta^p,
\]
for certain constants $C^5_{p,T},C^6_{p,T}$. Rearranging and using
Gr\"{o}nwall's inequality we obtain the final conclusion that
there exists a constant $C_{p,T}$ such that
\[
\EE\left[\sup_{0\leq t\leq T}\|e_t\|^p \right] \leq C_{p,T}\, \delta^p.
\]
\end{proof}


\section{Conclusions and future work}

The central conclusion from this paper is that by using an adaptive 
timestep it is possible to make the Euler-Maruyama approximation 
stable for SDEs with a globally Lipschitz volatility and a drift 
which is not globally Lipschitz but is locally Lipschitz and
satisfies a one-sided linear growth condition.
If the drift also satisfies a one-sided Lipschitz condition then
the order of strong convergence is $\halfs$, when looking at the 
accuracy versus the expected cost of each path.  For the important 
class of Langevin equations with unit volatility, the order of 
strong convergence is 1.

The numerical experiments suggest that in some applications 
the new method may not be significantly better than the tamed 
Euler-Maruyama method proposed and analysed by Hutzenthaler,
Jentzen \& Kloeden \cite{hjk12}, but in others it is shown to
be superior.

One direction for extension of the theory is to SDEs with a 
volatility which is not globally Lipschitz, but instead satisfies 
the Khasminskii-type condition used by Mao \& Szpruch 
\cite{mao15,ms13}.
Another is to extend the analysis to Milstein approximations,
which are particularly important when the SDE is scalar or 
satisfies the commutativity condition which means that the Milstein
approximation does not require the simulation of L{\'e}vy areas.
Another possibility is to use a Lyapunov function $V(x)$ in place
of $\|x\|^2$ in the stability analysis; this might enable one to 
prove stability and convergence for a larger set of SDEs.
For SDEs such as the stochastic van der Pol oscillator and the
stochastic Lorenz equation, if we could prove exponential 
integrability using the approach of Hutzenthaler, Jentzen \& Wang 
\cite{hjw16} then it may be possible to prove the order of 
strong convergence using a local one-sided Lipschitz condition.

A future paper will address a different challenge, extending 
the analysis to ergodic SDEs over an infinite time interval.
As well as proving a slightly different stability result with 
a bound which is uniform in time, the convergence analysis will 
show that under certain conditions the error bound is also 
uniformly bounded in time.  This is in contrast to the analysis 
in this paper in which the bound increases exponentially with 
time.

\bibliographystyle{plain}
\bibliography{mlmc,mc}

\begin{thebibliography}{10}

\bibitem{bs07}
J.W. Barrett and E.~S{\"u}li.
\newblock Existence of global weak solutions to some regularized kinetic models
  for dilute polymers.
\newblock {\em SIAM Multiscale Modelling and Simulation}, 6(2):506--546, 2007.

\bibitem{gl97}
J.G. Gaines and T.J. Lyons.
\newblock Variable step size control in the numerical solution of stochastic
  differential equations.
\newblock {\em SIAM Journal on Applied Mathematics}, 57(5):1455--1484, 1997.

\bibitem{giles08}
M.B. Giles.
\newblock Multilevel {M}onte {C}arlo path simulation.
\newblock {\em Operations Research}, 56(3):607--617, 2008.

\bibitem{giles15}
M.B. Giles.
\newblock Multilevel {M}onte {C}arlo methods.
\newblock {\em Acta Numerica}, 24:259--328, 2015.

\bibitem{glw16}
M.B. Giles, C.~Lester, and J.~Whittle.
\newblock Non-nested adaptive timesteps in multilevel {M}onte {C}arlo
  computations.
\newblock In R.~Cools and D.~Nuyens, editors, {\em Monte Carlo and Quasi-Monte
  Carlo Methods 2014}. Springer, 2016.

\bibitem{hms02}
D.J. Higham, X.~Mao, and A.M. Stuart.
\newblock Strong convergence of {E}uler-type methods for nonlinear stochastic
  differential equations.
\newblock {\em SIAM Journal on Numerical Analysis}, 40(3):1041--1063, 2002.

\bibitem{hmr01}
N.~Hofmann, T.~M{\"u}ller-Gronbach, and K.~Ritter.
\newblock The optimal discretization of stochastic differential equations.
\newblock {\em Journal of Complexity}, 17(1):117--153, 2001.

\bibitem{hj15}
M.~Hutzenthaler and A.~Jentzen.
\newblock {\em Numerical approximations of stochastic differential equations
  with non-globally {L}ipschitz continuous coefficients}, volume 236.
\newblock American Mathematical Society, 2015.

\bibitem{hjk11}
M.~Hutzenthaler, A.~Jentzen, and P.E. Kloeden.
\newblock Strong and weak divergence in finite time of {E}uler's method for
  stochastic differential equations with non-globally {L}ipschitz continuous
  coefficients.
\newblock {\em Proceedings of the Royal Society of London A: Mathematical,
  Physical and Engineering Sciences}, 467(2130):1563--1576, 2011.

\bibitem{hjk12}
M.~Hutzenthaler, A.~Jentzen, and P.E. Kloeden.
\newblock Strong convergence of an explicit numerical method for {SDE}s with
  nonglobally {L}ipschitz continuous coefficients.
\newblock {\em Annals of Applied Probability}, 22(4):1611--1641, 2012.

\bibitem{hjw16}
M.~Hutzenthaler, A.~Jentzen, and X.~Wang.
\newblock Exponential integrability properties of numerical approximation
  processes for nonlinear stochastic differential equations.
\newblock {\em Mathematics of Computation}, to appear, 2016.

\bibitem{kp92}
P.E. Kloeden and E.~Platen.
\newblock {\em Numerical Solution of Stochastic Differential Equations}.
\newblock Springer, Berlin, 1992.

\bibitem{lamba03}
H.~Lamba.
\newblock An adaptive timestepping algorithm for stochastic differential
  equations.
\newblock {\em Journal of Computational and Applied Mathematics},
  161(2):417--430, 2003.

\bibitem{lms07}
H.~Lamba, J.C. Mattingly, and A.M. Stuart.
\newblock An adaptive {E}uler-{M}aruyama scheme for {SDE}s: convergence and
  stability.
\newblock {\em IMA Journal of Numerical Analysis}, 27(3):479--506, 2007.

\bibitem{lemaire07}
V.~Lemaire.
\newblock An adaptive scheme for the approximation of dissipative systems.
\newblock {\em Stochastic Processes and their Applications},
  117(10):1491--1518, 2007.

\bibitem{mao97}
X.~Mao.
\newblock {\em Stochastic Differential Equations and Applications}.
\newblock Horwood Publishers Ltd., 1997.

\bibitem{mao15}
X.~Mao.
\newblock The truncated {E}uler-{M}aruyama method for stochastic differential
  equations.
\newblock {\em Journal of Computational and Applied Mathematics}, 290:370--384,
  2015.

\bibitem{mao16}
X.~Mao.
\newblock Convergence rates of the truncated {E}uler-{M}aruyama method for
  stochastic differential equations.
\newblock {\em Journal of Computational and Applied Mathematics}, 296:362--375,
  2016.

\bibitem{ms13}
X.~Mao and L.~Szpruch.
\newblock Strong convergence and stability of implicit numerical methods for
  stochastic differential equations with non-globally {L}ipschitz continuous
  coefficients.
\newblock {\em Journal of Computational and Applied Mathematics}, 238:14--28,
  2013.

\bibitem{mauthner98}
S.~Mauthner.
\newblock Step size control in the numerical solution of stochastic
  differential equations.
\newblock {\em Journal of Computational and Applied Mathematics},
  100(1):93--109, 1998.

\bibitem{mt05}
G.N. Milstein and M.V. Tretyakov.
\newblock Numerical integration of stochastic differential equations with
  nonglobally {L}ipschitz coefficients.
\newblock {\em SIAM Journal on Numerical Analysis}, 43(3):1139--1154, 2005.

\bibitem{muller04}
T.~M{\"u}ller-Gronbach.
\newblock Optimal pointwise approximation of {SDE}s based on {B}rownian motion
  at discrete points.
\newblock {\em Annals of Applied Probability}, pages 1605--1642, 2004.

\bibitem{wg13}
X.~Wang and S.~Gan.
\newblock The tamed {M}ilstein method for commutative stochastic differential
  equations with non-globally {L}ipschitz continuous coefficients.
\newblock {\em Journal of Difference Equations and Applications},
  19(3):466--490, 2013.

\end{thebibliography}

\end{document}